\documentclass[a4paper]{article}
\usepackage[T1]{fontenc}
\usepackage{lmodern}
\usepackage{graphicx}
\usepackage{authblk}
\usepackage{amsmath}
\usepackage{amsthm}
\usepackage{amssymb}
\usepackage{hyperref}
\usepackage{multirow}
\usepackage[letterpaper, left=1in, right=1in, bottom=1in, top=1in]{geometry}
\newtheorem{theorem}{Theorem}
\newtheorem{corollary}{Corollary}
\newtheorem{lemma}{Lemma}
\newtheorem{observation}{Observation}
\newtheorem{claim}{Claim}

\theoremstyle{definition}

\newtheorem{fact}{Fact}{\bfseries}{\itshape}
\newcommand{\RePair}{\mathsf{RePair}}
\newcommand{\opt}{\mathsf{Opt}}
\newcommand{\gfact}{\mathit{gfact}}
\newcommand{\PTree}{\mathcal{PT}}
\newcommand{\LZ}{\mathit{LZ}}
\newcommand{\SG}{\mathit{SG}}
\newcommand{\ReversePhi}[2]{\xi_{#1\rightarrow{#2}}}
\newcommand{\tta}{\mathtt{a}}
\newcommand{\ttb}{\mathtt{b}}

\begin{document}

\title{{RePair} Grammars are the Smallest Grammars for Fibonacci Words}
\author[1]{Takuya~Mieno}
\author[2,3]{Shunsuke~Inenaga}
\author[1]{Takashi~Horiyama}
\affil[1]{Faculty of Information Science and Technology, Hokkaido University.\par\texttt{\{takuya.mieno,horiyama\}@ist.hokudai.ac.jp}}
\affil[2]{Department of Informatics, Kyushu University.\par\texttt{inenaga@inf.kyushu-u.ac.jp}}
\affil[3]{PRESTO, Japan Science and Technology Agency.}
\date{}
\maketitle
\begin{abstract}
  Grammar-based compression is a loss-less data compression scheme that represents a given string $w$ by a context-free grammar that generates only $w$.
  While computing the smallest grammar which generates a given string $w$ is NP-hard in general, a number of polynomial-time grammar-based compressors which work well in practice have been proposed.
  \emph{RePair}, proposed by Larsson and Moffat in 1999, is a grammar-based compressor which recursively replaces all possible occurrences of a most frequently occurring bigrams in the string.
  Since there can be multiple choices of the most frequent bigrams to replace, different implementations of RePair can result in different grammars.
  In this paper, we show that the smallest grammars generating the Fibonacci words $F_k$ can be completely characterized by RePair, where $F_k$ denotes the $k$-th Fibonacci word.
  Namely, all grammars for $F_k$ generated by any implementation of RePair are the smallest grammars for $F_k$, and no other grammars can be the smallest for $F_k$.
  To the best of our knowledge, Fibonacci words are the first non-trivial infinite family of strings for which RePair is optimal.
\end{abstract}
 \section{Introduction}
A context-free grammar in the Chomsky normal form that produces only a single string $w$ is called a \emph{straight-line program} (\emph{SLP}) for $w$.
Highly repetitive strings that contain many long repeats can be compactly represented by SLPs since occurrences of equal substrings can be replaced by a common non-terminal symbol.
\emph{Grammar-based compression} is a loss-less data compression scheme that represents a string $w$ by an SLP for $w$.
We are aware of more powerful compression schemes such as
run-length SLPs~\cite{Jez15,NishimotoIIBT16,BilleGGP18},
composition systems~\cite{GasieniecKPR96},
collage systems~\cite{KidaMSTSA03},
NU-systems~\cite{NavarroU21},
the Lempel-Ziv 77 family~\cite{LZ77,StorerS82,Crochemore84,FarachT98},
and bidirectional schemes~\cite{StorerS82}.
Nevertheless, since SLPs exhibit simpler structures than those,
a number of efficient algorithms that can work directly on SLPs
have been proposed, including
pattern matching~\cite{KarpinskiRS97,Jez15},
convolutions~\cite{TanakaIIBT13},
random access~\cite{BilleLRSSW15},
detection of repeats and palindromes~\cite{I2015},
Lyndon factorizations~\cite{INIBT16},
longest common extension queries~\cite{I17},
longest common substrings~\cite{MatsubaraIISNH09},
finger searches~\cite{BilleCCG18},
and balancing the grammar~\cite{ganardi2021balancing}.
More examples of algorithms directly working on SLPs can be found in references therein and the survey~\cite{SLPSurvey2012}.
Since these algorithms do not decompress the SLPs,
they can be more efficient than solutions on uncompressed strings.

Since the complexities of the algorithms mentioned above depend on the size of the SLP, it is important to compute a small grammar for a given string.
\emph{The smallest grammar problem} is to find a grammar
that derives a given string $w$, where the total length of the right-hand sides of the productions is the smallest possible.
The smallest grammar problem is known to be NP-hard in general~\cite{StorerS82,Charikar_smallestgrammar_2005}.
Namely, there is no polynomial-time algorithm that finds the smallest grammar for \emph{arbitrary} strings, unless P = NP.
Notably, the NP-hardness holds even when the alphabet size is bounded by some constant at least $17$~\cite{Casel2021},
on the other hand, it is open whether the NP-hardness holds for strings over a smaller constant alphabet, particularly on binary alphabets.

We consider a slightly restricted version of the smallest grammar problem where the considered grammars are SLPs, i.e., only those in the Chomsky normal form.
We follow a widely accepted definition for the \emph{size} of an SLP,
which is the number of productions in it.
Thus, in the rest of our paper, grammars mean SLPs unless otherwise stated,
and our smallest grammar problem seeks the smallest SLP, which generates the input string with the fewest productions\footnote{There is an alternative definition of the size of a grammar, that is, the total sum of the lengths of the right side of its rules. This definition is usually used for non-SLP grammars.}.
There are some trivial examples of strings whose smallest grammar sizes can be easily determined, e.g.,
 a unary string $(\tta)^{2^i}$ of length power of two\footnote{
  Grammars for unary words are closely related to \emph{addition chains}~\cite{Knuth98},
  and the smallest (not necessarily SLP) grammar for $(\tta)^k$ is non-trivial for general $k$ that is not a power of two.
  Also, in such a case, RePair does not provide the smallest grammar for $(\tta)^k$~\cite{HuckeR21}.
},
and non-compressible strings in which all the symbols are distinct.
It is interesting to identify classes of strings whose smallest grammars can be determined in polynomial-time
since it may lead to more and deeper insights to the smallest grammar problem.
To the best of our knowledge, however, no previous work shows non-trivial strings whose smallest grammar sizes are computable in polynomial-time.

In this paper, we study the smallest grammars of the \emph{Fibonacci words} $\{F_1, F_2, \ldots, F_n, \ldots\}$
defined recursively as follows: $F_1 = \ttb$, $F_2 = \tta$, and $F_i=F_{i-1}F_{i-2}$ for $i \ge 3$.
We show that the smallest grammars of the Fibonacci words can be completely characterized by the famous \emph{RePair}~\cite{RePair} algorithm,
which is the best known practical grammar compressor that recursively replaces the most frequently occurring bigram with a new non-terminal symbol in linear total time.
We first prove that the size of the smallest grammar of the $n$-th Fibonacci word $F_n$ is $n$.
We then prove that applying \emph{any implementation} of RePair to $F_n$ always provides a smallest grammar of $F_n$,
and conversely, only such grammars can be the smallest for Fibonacci words.
This was partially observed earlier in the experiments by Furuya et al.~\cite{FuruyaTNIBK20}, where five different implementations of RePair produced grammars of the same size for the fib41 string from the Repetitive Corpus of the Pizza\&Chili Corpus (\url{http://pizzachili.dcc.uchile.cl/repcorpus.html}).
However, to our knowledge, this paper is the first that gives theoretical evidence.

\paragraph*{Related Work.}
Although the smallest grammar problem is NP-hard,
there exist polynomial-time approximations to the problem:
Rytter's AVL-grammar~\cite{Rytter2003LZandGrammar} produces an SLP of size $O(s^* \log (N / s^*))$, where $s^*$ denote the size of the smallest SLP for the input string and $N$ is the length of the input string.
The $\alpha$-balanced grammar of Charikar et al.~\cite{Charikar_smallestgrammar_2005} produces a (non-SLP) grammar of size $O(g^* \log (N / g^*))$,
where $g^*$ denotes the size of the smallest (non-SLP) grammar.
Upper bounds and lower bounds for the approximation ratios of other practical grammar compressors
including LZ78~\cite{LZ78}, BISECTION~\cite{KiefferYNC00}, RePair~\cite{RePair}, SEQUENTIAL~\cite{YangK00}, LONGEST MATCH~\cite{KiefferYNC00}, and GREEDY~\cite{ApostolicoL00}, are also known~\cite{Charikar_smallestgrammar_2005,BannaiHHIJLR21}.
Charikar et al.~\cite{Charikar_smallestgrammar_2005} showed that 
the approximation ratio of RePair to the smallest (non-SLP) grammar is at most $O((N/\log N)^{2/3})$ and is at least $\Omega(\sqrt{\log N})$.
The lower bound was later improved by Bannai et al.~\cite{BannaiHHIJLR21} to $\Omega(\log N/\log\log N)$.
Furthermore, it is known that RePair has a lower bound on the approximation ratio $\log_2(3)$ to the smallest (non-SLP) grammar for \emph{unary} strings~\cite{HuckeR21}.
On the other hand, RePair is known to achieve the best compression ratio on many real-world datasets and enjoy applications in web graph compression~\cite{ClaudeN10} and XML compression~\cite{LohreyMM13}.
Some variants of RePair have also been proposed~\cite{MasakiK16,BilleGP17,GanczorzJ17,GagieIMNST19,FuruyaTNIBK20,KopplIFTSG21}.

 \section{Preliminaries}\label{sec:pre}

\subsection{Strings}
Let $\Sigma$ be an alphabet.
An element in $\Sigma$ is called a symbol.
An element in $\Sigma^\ast$ is called a string.
The length of string $w$ is denoted by $|w|$.
The empty string $\varepsilon$ is the string of length $0$.
For each $i$ with $1 \le i \le |w|$, $w[i]$ denotes the $i$-th symbol of $w$.
For each $i$ and $j$ with $1 \le i \le j \le |w|$, $w[i.. j]$ denotes the \emph{substring} of $w$ which begins at position $i$ and ends at position $j$.
For convenience, let $w[i..j] = \varepsilon$ if $i > j$.
When $i = 1$~(resp.~$j = |w|$), $w[i.. j]$ is called a \emph{prefix} (resp. a \emph{suffix}) of $w$.
For non-empty strings $w$ and $b$ with $|b| < |w|$, $b$ is called a \emph{border} of $w$ if $b$ is both a prefix and a suffix of $w$.
If there are no borders of $w$, then $w$ is said to be \emph{borderless}.
For any non-empty string $w$, we call $w[|w|]w[1..|w|-1]$ the \emph{right-rotation} of $w$.
For a string $w$, $\sigma_w$ denotes the number of distinct symbols appearing in $w$.
For a non-empty string $w$, we denote by $w^R$ the \emph{reversed string} of $w$,
  namely $w^R = w[|w|] \cdots w[1]$.

\subsection{Fibonacci Words and Related Words}
For a binary alphabet $\{a, b\}$,
\emph{Fibonacci words} $F^{(a, b)}_i$ (starting with $a$ for $i > 1$) are defined as follows:
$F^{(a, b)}_1 = b$, $F^{(a, b)}_2 = a$, and $F^{(a, b)}_i = F^{(a, b)}_{i-1}F^{(a, b)}_{i-2}$ for $i \ge 3$.
We call $F^{(a, b)}_i$ the $i$-th Fibonacci word (starting with $a$ for $i > 1$).
By the above definition of Fibonacci words, $|F^{(a,b)}_i| = f_i$ holds for each $i$,
where $f_i$ denotes the $i$-th Fibonacci number defined as follows: $f_1 = 1$, $f_2 = 1$, $f_i = f_{i-1} + f_{i-2}$ for $i \ge 3$.
There is an alternative definition~(e.g.~\cite{Berstel1986}) of Fibonacci words using the \emph{string morphism} $\phi^{(a, b)}$:
The $i$-th Fibonacci word $F^{(a, b)}_i$ (starting with $a$ for $i > 1$) is $(\phi^{(a, b)})^{i-1}({b})$, where $\phi^{(a, b)}$ is a morphism over $\{{a}, {b}\}$
such that $\phi^{(a, b)}({a}) = {ab}$ and $\phi^{(a, b)}({b}) = {a}$.
We strictly distinguish the morphism $\phi^{(b,a)}$ from $\phi^{(a,b)}$ over the same binary alphabet $\{a, b\}$, namely, $\phi^{(b,a)}$ generates 
the Fibonacci words $F^{(b, a)}_i$ where $a$ and $b$ are flipped in $F^{(a,b)}_i$.
We will omit the superscript $(a, b)$ if it is clear from contexts or it is not essential for the discussion.

Next, we define other words, which will be utilized to analyze the smallest grammar of Fibonacci words.
Let $\pi^{(a, b)}$ be the morphism over $\{a, b\}$ such that $\pi^{(a, b)}({a}) = {ab}$ and $\pi^{(a, b)}({b}) = {abb}$.
Further, let $\theta^{(a, b)}$ be the morphism over $\{a, b\}$ such that $\theta^{(a, b)}({a}) = {aab}$ and $\theta^{(a, b)}({b}) = {ab}$.
For each positive integer $i$, we define $P_i^{(a,b)}$ and $Q_i^{(a,b)}$ over $\{a,b\}$
by $P_i^{(a,b)} = (\pi^{(a, b)})^{i-1}(a)$ and $Q_i^{(a,b)} = (\theta^{(a, b)})^{i-1}(a)$, respectively.
We treat their superscripts as for that of Fibonacci words.
We will later show 
that $|P_{i}| = |F_{2i-1}| = f_{2i-1}$ and $|Q_{i}| = |F_{2i}| = f_{2i}$ for any $i \geq 1$.
We show examples for these three words in Table~\ref{tab:example_F_P_Q}.
We remark that strings $P_i$ and $Q_i$ can be obtained at some point while RePair is being applied to the Fibonacci words.
  We will prove this in Section~\ref{sec:RePairGrammar}.

For a symbol $X$ and a string $y$, let $\ReversePhi{X}{y}$ be the morphism
such that $\ReversePhi{X}{y}(X) = y$ and $\ReversePhi{X}{y}(c) = c$ for any symbol $c \ne X$.
Namely, when applied to a string $w$, $\ReversePhi{X}{y}(w)$ replaces all occurrences of $X$ in $w$ with $y$ but any other symbols than $X$ remain unchanged.
For any morphism $\lambda$ and any sequence $S = (s_1, \ldots, s_m)$ of strings,
let $\lambda(S) = (\lambda(s_1), \ldots, \lambda(s_m))$.

\begin{table}[tb]
\centering
\begin{tabular}{c|l|c}
\hline
$i$ & $F_i^{(\mathtt{a}, \mathtt{b})}$ & length\\
\hline
1 & $\mathtt{b}$     & 1\\
2 & $\mathtt{a}$     & 1\\
3 & $\mathtt{ab}$    & 2\\
4 & $\mathtt{aba}$   & 3\\
5 & $\mathtt{abaab}$ & 5\\
6 & $\mathtt{abaababa}$ & 8\\
7 & $\mathtt{abaababaabaab}$ & 13\\
8 & $\mathtt{abaababaabaababaababa}$ & 21\\
9 & $\mathtt{abaababaabaababaababaabaababaabaab}$ & 34\\
10 & $\mathtt{abaababaabaababaababaabaababaabaababaababaabaababaababa}$ & 55\\
\hline
\hline
$i$ & $P_i^{(\mathtt{a}, \mathtt{b})}$ & length\\
\hline
1 & $\mathtt{a}$                                  & 1\\
2 & $\mathtt{ab}$                                 & 2\\
3 & $\mathtt{ababb}$                              & 5\\
4 & $\mathtt{ababbababbabb}$                      & 13\\
5 & $\mathtt{ababbababbabbababbababbabbababbabb}$ & 34\\
\hline
\hline
$i$ & $Q_i^{(\mathtt{a}, \mathtt{b})}$ & length\\
\hline
1 & $\mathtt{a}$                                                       & 1\\
2 & $\mathtt{aab}$                                                     & 3\\
3 & $\mathtt{aabaabab}$                                                & 8\\
4 & $\mathtt{aabaababaabaababaabab}$                                   & 21\\
5 & $\mathtt{aabaababaabaababaababaabaababaabaababaababaabaababaabab}$ & 55\\
\hline
\end{tabular}
\caption{
 Lists of $F_i^{(\mathtt{a}, \mathtt{b})}$ for $i = 1, \ldots, 10$, and
 $P_i^{(\mathtt{a}, \mathtt{b})}$ and $Q_i^{(\mathtt{a}, \mathtt{b})}$ for $i = 1, \ldots, 5$.
}\label{tab:example_F_P_Q}
\end{table}

\subsection{Grammar Compression and RePair}

A context-free grammar in the Chomsky normal form
that produces a single string $w$ is called a \emph{straight-line program} (\emph{SLP} in short) for $w$.
Namely, any production in a grammar is of form either $X_i\rightarrow\alpha$ or $X_i\rightarrow X_jX_k$,
where $\alpha$ is a terminal symbol and
$X_i, X_j$, and $X_k$ are non-terminal symbols such that $i > j$ and $i > k$, that is, there are no \emph{cycles} in the productions.
In what follows, we refer to an SLP that produces $w$ simply as a \emph{grammar of $w$}.
Let $\mathcal{T}(G)$ denote the derivation tree of a grammar $G$,
where each internal node in $\mathcal{T}(G)$ is labeled by the corresponding non-terminal symbol of $G$.
As in~\cite{Rytter2003LZandGrammar},
we conceptually identify terminal symbols with their parents
so that $\mathcal{T}(G)$ is a full binary tree (i.e. every internal node has exactly two children).
Let $G_1$ and $G_2$ be grammars both deriving the same string $w$,
and let $\Pi_1$ and $\Pi_2$ be the sets of non-terminal symbols of $G_1$ and $G_2$, respectively.
We say that $G_1$ and $G_2$ are \emph{equivalent}
if there exists a renaming bijection $f:\Pi_1 \rightarrow \Pi_2$
that transforms $\mathcal{T}(G_1)$ to $\mathcal{T}(G_2)$.
We say that $G_1$ and $G_2$ are \emph{distinct} if they are not equivalent.
For example, two grammars
$\{A \rightarrow a, B \rightarrow b, C \rightarrow AB, D \rightarrow CA\}$
and
$\{X \rightarrow a, Y \rightarrow b, Z \rightarrow XY, W \rightarrow ZX\}$
are equivalent grammars both deriving string $aba$.

Equivalent grammars form an equivalence class of grammars,
and we pick an arbitrary one as the representative of each equivalence class.
A \emph{set} $\mathsf{S}$ of grammars that derive the same string $w$
is a set which consists of (some) representative grammars, which means that any two grammars in $\mathsf{S}$ are mutually distinct.
See Figure~\ref{fig:grammar} for examples of distinct grammars for the same string.
\begin{figure}[tb]
  \centerline{\includegraphics[width=\linewidth]{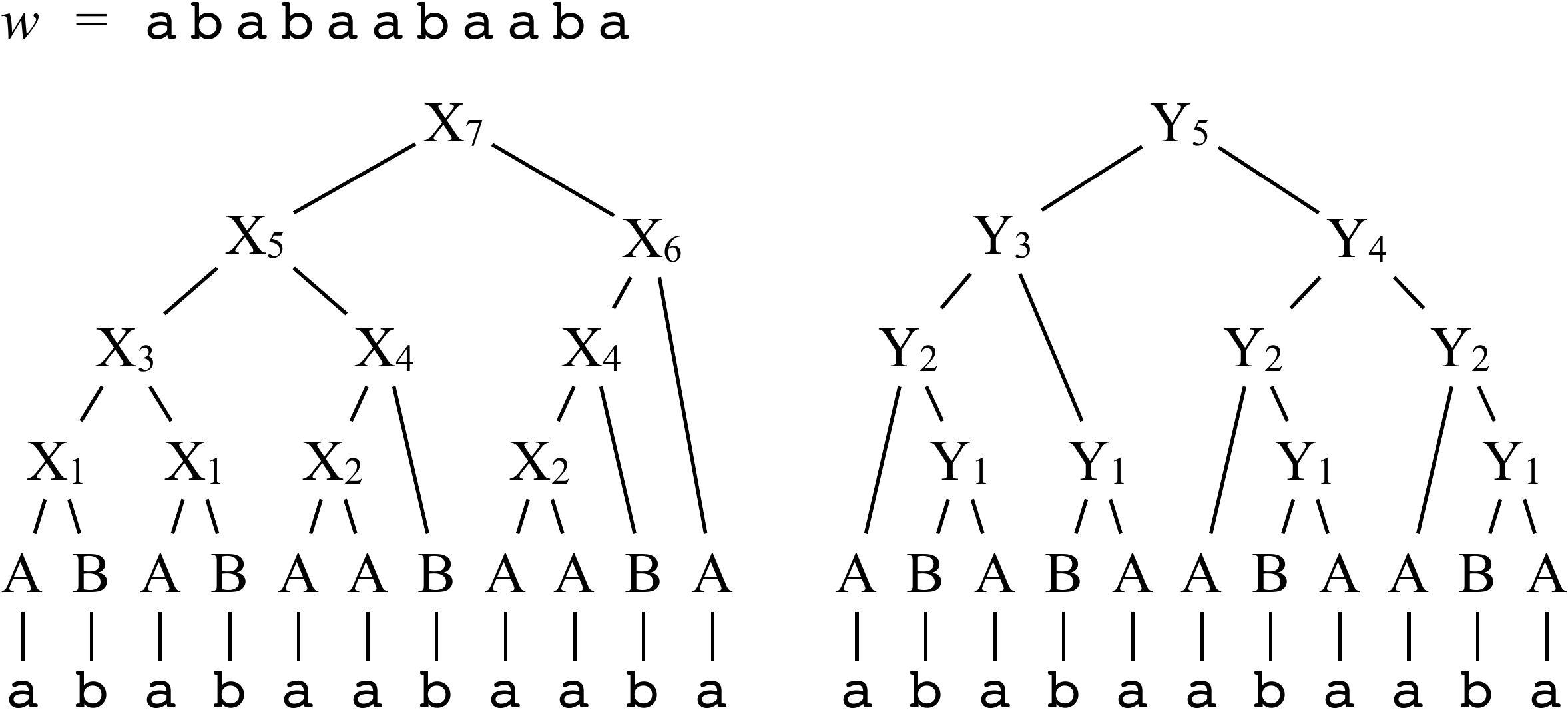}}
  \caption{
    Illustration for the derivation trees of two distinct grammars of string $w = \mathtt{ababaabaaba}$.
    The size of the grammar on the left is $9$ since there are nine productions;
    $\{
      A \rightarrow \tta,
      B \rightarrow \ttb,
      X_1 \rightarrow AB,
      X_2 \rightarrow AA,
      X_3 \rightarrow X_1X_1,
      X_4 \rightarrow X_2B,
      X_5 \rightarrow X_3X_4,
      X_6 \rightarrow X_4A,
      X_7 \rightarrow X_5X_6,
    \}$.
    On the other hand, the size of the grammar on the right is $7$.
    Note that the right one is a RePair grammar of $w$.
    In the rest of the paper, we sometimes identify the terminal symbols (leaves) with their parents so the derivation trees are (conceptually) full binary trees.
  }
  \label{fig:grammar}
\end{figure}

The \emph{size} of a grammar $G$, denoted by $|G|$, is the number of productions in $G$.
We denote by $g^\ast(w)$ the size of the smallest grammar of string $w$.
Further, we denote by $\opt(w)$ the set of all the smallest grammars of string $w$.
While computing $g^\ast(w)$ for a given string $w$ is NP-hard in general~\cite{Charikar_smallestgrammar_2005},
a number of practical algorithms which run in polynomial-time and construct small grammars of $w$ have been proposed.

In this paper, we focus on \emph{RePair}~\cite{RePair}, which is the best known grammar-based compressor that produce small grammars in practice.
We briefly describe the RePair algorithm, which consists of the three stages:
\begin{enumerate}
  \item Initial stage: All terminal symbols in the input string are replaced with non-terminal symbols. This creates unary productions.

  \item Replacement stage: The algorithm picks an \emph{arbitrary} bigram which has the most non-overlapping occurrences in the string,
and then replaces all possible occurrences of the bigram with a new non-terminal symbol.
The algorithm repeats the same process recursively for the string obtained after the replacement of the bigrams, until no bigrams have two or more non-overlapping occurrences in the string.
It is clear that the productions created in the replacement stage are all binary.

  \item Final stage: Trivial binary productions are created from the sequence of non-terminal symbols which are obtained after the last replacement.
This ensures that the resulting grammar is in the Chomsky normal form.
We remark that when distinct bigrams have the most non-overlapping occurrences in the string, then the choice of the bigrams to replace depends on each implementation of RePair.
\end{enumerate}
A grammar of $w$ obtained by some implementation of RePair is called a \emph{RePair grammar} of $w$.
We denote by $\RePair(w)$ the set of all possible RePair grammars of $w$.
We show an example of RePair grammars in Figure~\ref{fig:RePair_grammars}.
\begin{figure}[tb]
  \centerline{\includegraphics[width=\linewidth]{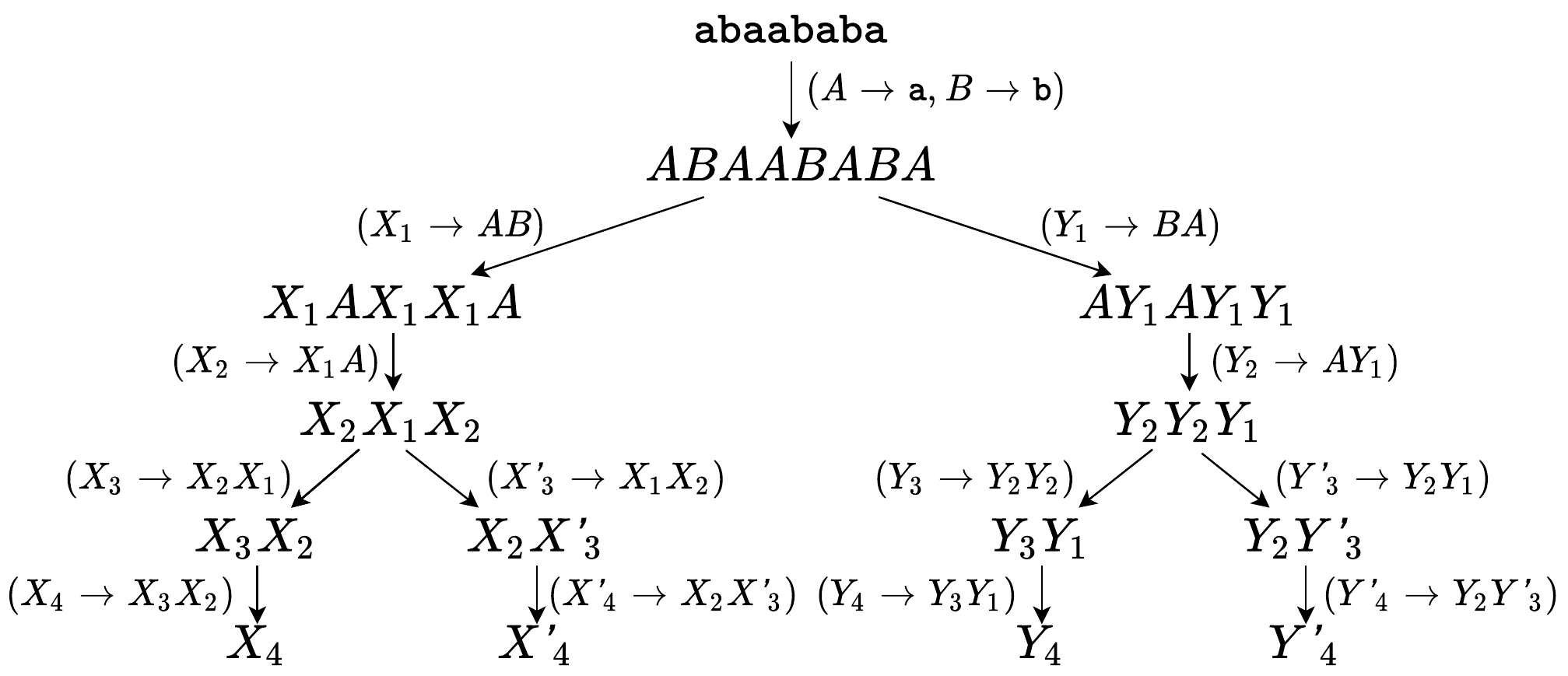}}
  \caption{
    Illustration for the changes of strings and productions to be added
    when RePair is applied to string $w = \mathtt{abaababa}$.
    At the second level, the most frequent bigrams in string $ABAABABA$ are $AB$ and $BA$.
    If $AB$ is chosen and replaced with non-terminal symbol $X_1$,
    the string changes to $X_1AX_1X_1A$ and production $X_1\rightarrow AB$ is added.
    Otherwise (if $BA$ is chosen and replaced with non-terminal symbol $Y_1$),
    the string changes to $AY_1AY_1Y_1$ and production $Y_1\rightarrow BA$ is added.
    In this example, the size of $\RePair(w)$ is four.
  }
  \label{fig:RePair_grammars}
\end{figure}

\subsection{LZ-factorization}
A sequence $S = (s_1, \ldots, s_m)$ of non-empty strings is called a \emph{factorization} of string $w$ if $w = s_1\cdots s_m$.
Each $s_i$~($1 \leq i \leq m$) is called a \emph{phrase} of $S$.
The \emph{size} of the factorization $S$, denoted $|S|$,
is the number $m$ of phrases in $S$.

For a factorization $S = (s_1, \ldots, s_m)$ of a string $w$, we say that the $i$-th phrase $s_i$ is \emph{greedy}
if either $s_i$ is a fresh symbol that occurs for the first time in $s_1 \cdots s_i$,
or $s_i$ is the longest prefix of $s_i\cdots s_m$ which occurs in $s_1\cdots s_{i-1}$.
A factorization of a string $w$ is called the \emph{LZ-factorization} of $w$ if all the phrases are greedy.
Note that this definition of the LZ-factorization is equivalent to the one in~\cite{Rytter2003LZandGrammar}.
The LZ-factorization of string $w$ is denoted by $\LZ(w)$, and the size of $\LZ(w)$ is denoted by $z(w)$.
We sometimes represent a factorization $(s_1, s_2, \ldots, s_m)$ of $w$ by $s_1|s_2|\ldots|s_m$, where each $|$ denotes the \emph{boundary} of the phrases.
For example,
The LZ-factorization of $w = \mathtt{ababaabaaba}$ is $\mathtt{a|b|ab|a|aba|aba}$.
 \section{Basic Properties of Fibonacci and Related Words}
In this section, we show some properties of the aforementioned words.
We fix the alphabet $\Sigma = \{\tta, \ttb\}$ in this section.
First, for the summation of Fibonacci sequences, the next equations hold:
\begin{fact} \label{fact:fib}
  $\sum_{k=1}^{i}f_{2k-1} = f_{2i}$ and $\sum_{k=1}^{i}f_{2k} = f_{2i+2}-1$.
\end{fact}
By the definitions of $F_n$, $P_n$, and $Q_n$, we have the following observation:
\begin{lemma}\label{lem:freq}
  For each $k \ge 2$,
  the most frequent bigrams of $F_{2k}$ are $\mathtt{ab}$ and $\mathtt{ba}$, and
  the most frequent bigram of $F_{2k-1}$ is $\mathtt{ab}$.
  Also, for each $i \ge 2$ and each $j \ge 3$, the most frequent bigram of $P_i$ and $Q_j$ is $\mathtt{ab}$.
\end{lemma}
\begin{proof}
From the fact that bigram $\mathtt{bb}$ and trigram $\mathtt{aaa}$ do not occur in any Fibonacci word (e.g., see \cite{Berstel1986}),
  we can see that any occurrence of $\mathtt{aa}$ is followed by $\mathtt{b}$ in Fibonacci words.
  Thus, $\mathtt{aa}$ cannot occur more frequently than $\mathtt{ab}$ in any Fibonacci word.
  Also, since the third and subsequent Fibonacci words start with $\mathtt{ab}$,
  bigram $\mathtt{ab}$ occurs more frequently than $\mathtt{aa}$.
  Additionally, since all the Fibonacci words $F_{2k}$ of even order end with $\ttb$ and all the Fibonacci words $F_{2k-1}$ of odd order end with $\tta$,
  the statements for the Fibonacci words hold.

  Similarly, as for string $P_i$, it follows from the definition of morphism $\pi$ that $\mathtt{aa}$ does not occur in $P_i$. 
  Also, $\mathtt{bb}$ always succeeds $\tta$, and thus, $\mathtt{bb}$ cannot occur more frequently than $\mathtt{ab}$.
  Furthermore, by the definition of morphism $\pi$, string $P_i$ starts with $\mathtt{aba}$ and ends with $\ttb$ for $i \ge 3$.
  Thus, the most frequent bigram of $P_i$ is $\mathtt{ab}$ for $i \ge 3$~(note that $P_2$ is trivial).

  Finally, as for string $Q_j$, it follows from the definition of morphism $\theta$ that $\mathtt{bb}$ does not occur in $Q_j$.
  Also, $\mathtt{aa}$ always precedes $\ttb$, and thus, $\mathtt{aa}$ cannot occur more frequently than $\mathtt{ab}$.
  Furthermore, by the definition of morphism $\theta$, string $Q_j$ ends with $\mathtt{ab}$ for $j \ge 3$.
  Thus, the most frequent bigram of $Q_j$ is $\mathtt{ab}$ for $j \ge 3$.
\end{proof}

A factorization $C = (c_1, \ldots, c_m)$ of a string $w$ is called the \emph{C-factorization} of $w$ if
either $c_i$ is a fresh symbol or $c_i$ is the longest prefix of $c_i\cdots c_m$ which occurs twice in $c_1\cdots c_i$.
We can obtain the full characterization of the LZ-factorization of $F_n$ immediately from the C-factorization of $F_n$, as follows:
\begin{lemma} \label{lem:LZ_of_F}
  The LZ-factorization of $F_n$ is $(\tta, \ttb, \tta, F_4^R, \ldots, F_{n-2}^R, s)$,
  where $s = \mathtt{ab}$ if $n$ is odd, and $s = \mathtt{ba}$ otherwise.
\end{lemma}
\begin{proof}
  It is shown in \cite{Berstel_2006} that 
  the C-factorization of the infinite Fibonacci word $\mathbf{F}$ is $(\tta, \ttb, \tta, F_4^R, F_5^R, \ldots)$.
  Also, for each $i \ge 4$, the (only) reference source of each factor $F_i^R$ is the substring of $F_n$ of length $f_i$ ending at just before the factor, i.e., the source does not overlap with the factor.
    From these facts, it can be seen that the LZ-factorization of $\mathbf{F}$ is the same as the C-factorization of $\mathbf{F}$.
  Then, the last phrase of the C-factorization of a finite Fibonacci word is of length two
  since $f_n = \sum_{i=1}^{n-2}f_{i} + 1 = (1 + 1 + 2 + \sum_{i=4}^{n-2}f_{i}) + 1 = 3 + \sum_{i=4}^{n-2}f_{i} + 2$.
  Also, since the Fibonacci words of odd order (resp. even order) end with $\tta\ttb$~(resp. $\ttb\tta$),
  the last phrase is $\mathtt{ab}$~(resp. $\mathtt{ba}$).
\end{proof}

The next lemma states that $P_i$ and $Q_i$ are the right-rotations of Fibonacci words.
\begin{lemma}\label{lem:P_Q_are_rotFib}
  For each $i \ge 1$,
  $P_i^{(\tta, \ttb)}$ is the right-rotation of $F_{2i-1}^{(\ttb, \tta)}$, and
  $Q_i^{(\tta, \ttb)}$ is the right-rotation of $F_{2i}^{(\tta, \ttb)}$.
\end{lemma}
\begin{proof}
The next claim can be proven by induction:
  \begin{claim}\label{claim:rotation}
    For any non-empty string $x \in \{\tta, \ttb\}$, 
    $(\phi^{(\ttb, \tta)})^2(x)\ttb = \ttb\pi^{(\tta, \ttb)}(x)$ and\linebreak
    $(\phi^{(\tta, \ttb)})^2(x)\tta\ttb = \tta\ttb\theta^{(\tta, \ttb)}(x)$ hold.
  \end{claim}
  We prove the lemma by using Claim~\ref{claim:rotation}.
  Assume that the lemma holds for $i$.
  Since the last symbol of $F_{2i-1}^{(\ttb, \tta)}$ is $\tta$,
  we can write $F_{2i-1}^{(\ttb, \tta)} = x\tta$ with some string $x$.
  From the induction hypothesis, $P_i^{(\tta, \ttb)} = \tta x$ holds.
  Then, $P_{i+1}^{(\tta, \ttb)} = \pi^{(\tta, \ttb)}(\tta x) = \tta\ttb\pi^{(\tta, \ttb)}(x)$.
  Also, $F_{2i+1}^{(\ttb, \tta)} = (\phi^{(\ttb, \tta)})^2(x\tta) = (\phi^{(\ttb, \tta)})^2(x)\ttb\tta = \ttb\pi^{(\tta, \ttb)}(x)\tta$ by Claim~\ref{claim:rotation},
  and thus, $P_{i+1}^{(\tta, \ttb)}$ is the right-rotation of $F_{2i+1}^{(\ttb, \tta)}$.
Similarly, since the last symbol of $F_{2i}^{(\tta, \ttb)}$ is $\tta$,
  we can write $F_{2i}^{(\tta, \ttb)} = y\tta$ with some string $y$.
  From the induction hypothesis, $Q_i^{(\tta, \ttb)} = \tta y$ holds.
  Then, $Q_{i+1}^{(\tta, \ttb)} = \theta^{(\tta, \ttb)}(\tta y) = \tta\tta\ttb\theta^{(\tta, \ttb)}(y)$.
  Also, $F_{2i+2}^{(\tta, \ttb)} = (\phi^{(\tta, \ttb)})^2(y\tta) = (\phi^{(\tta, \ttb)})^2(y)\tta\ttb\tta = \tta\ttb\theta^{(\tta, \ttb)}(y)\tta$ by Claim~\ref{claim:rotation},
  and thus, $Q_{i+1}^{(\tta, \ttb)}$ is the right-rotation of $F_{2i+2}^{(\tta, \ttb)}$.
\end{proof}

The LZ-factorizations of $P_n$ is as follows:
\begin{lemma}\label{lem:lz_of_P}
  Let $i \ge 2$.
  The $j$-th phrase of $\LZ(P_i^{(\tta, \ttb)})$ is $(F_j^{(\ttb, \tta)})^R$ for each $j$ with $1 \le j \le 2i-3$.
  The last phrase is the $(2i-2)$-th phrase, and that is $\ttb$.
\end{lemma}
\begin{proof}
  The first three phrases of  $\LZ(P_i)$ are $(\tta, \tta, \mathtt{ab}) = ((F_1^{(\ttb, \tta)})^R, (F_2^{(\ttb, \tta)})^R, (F_3^{(\ttb, \tta)})^R)$.
  By Lemma~\ref{lem:P_Q_are_rotFib}, $P_i[2..|P_i|]\tta = F_{2i-1}^{(\ttb, \tta)}$.
  Namely, $\LZ(P_i[2..|P_i|]\tta) = \LZ(F_{2i-1}^{(\ttb, \tta)})$.
By Lemma~\ref{lem:LZ_of_F}, for each phrase of length at least four of $\LZ(F_{2i-1}^{(\ttb, \tta)})$,
  the length-$3$ prefix of the phrase is either $\mathtt{abb}$ or $\mathtt{bab}$.
  Since $P_i$ starts with $\mathtt{aba}$,
  for each $k \ge 4$, the $k$-th phrase of $\LZ(P_i\tta)$ equals that of $\LZ(P_i[2.. |P_i|]\tta) = \LZ(F_{2i-1}^{(\ttb, \tta)})$,
  that is, $(F_k^{(\ttb, \tta)})^R$~(see also Figure~\ref{fig:LZ_P}).
  Also, $|\LZ(P_i)| = |\LZ(F_{2i-1}^{(\ttb, \tta)})| = 2i-2$ holds, and the last phrase of $\LZ(P_i)$ is $\ttb$.
\begin{figure}[tb]
  \centerline{\includegraphics[width=\linewidth]{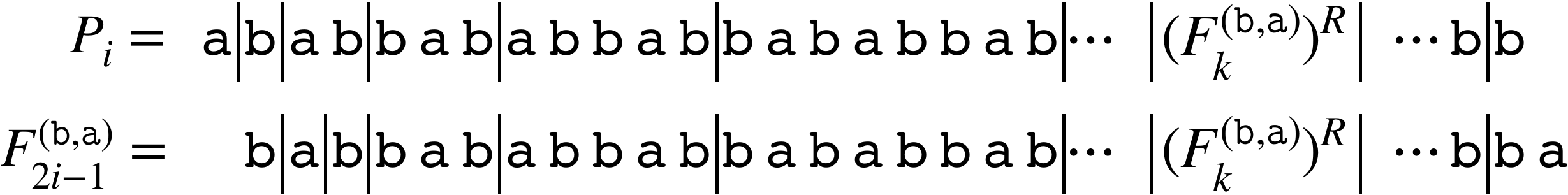}}
  \caption{
    Illustration for the LZ-factorizations of $P_i$ and $F_{2i-1}^{(\ttb, \tta)}$.
    $F_{2i-1}^{(\ttb, \tta)}$ is aligned so that the first position of $F_{2i-1}^{(\ttb, \tta)}$ is the second of $P_i$.
    In both of two factorizations, the $k$-th phrase is $(F_{k}^{(\ttb, \tta)})^R$ for each $k \ge 4$.
  }
  \label{fig:LZ_P}
\end{figure}
\end{proof}
 \section{RePair Grammars of Fibonacci Words}\label{sec:RePairGrammar}
In this section, we first show a lower bound of the size of the smallest grammar of any string,
which is slightly tighter than the well-known result shown by Rytter~\cite{Rytter2003LZandGrammar}.
Second, we show that the size of RePair grammars of Fibonacci words are always the smallest.

\subsection{Tighter Lower Bound of Smallest Grammar Size}

The \emph{partial derivation tree} $\PTree(G)$ of a grammar $G$ is the maximal subgraph of the derivation tree of $G$ such that
for each non-leaf node $v$ in $\PTree(G)$, there is no node whose label is the same as $v$ to its left.
For a grammar $G$ of a string $w$, the \emph{g-factorization} of $w$ w.r.t.~$G$, denoted by $\gfact(G)$,
is the factorization of $w$ where the phrases correspond to the leaves of $\PTree(G)$.
See also Figure~\ref{fig:gfact} for an example of $\PTree(w)$ and $\gfact(G)$.
\begin{figure}[tb]
  \centerline{\includegraphics[width=0.4\linewidth]{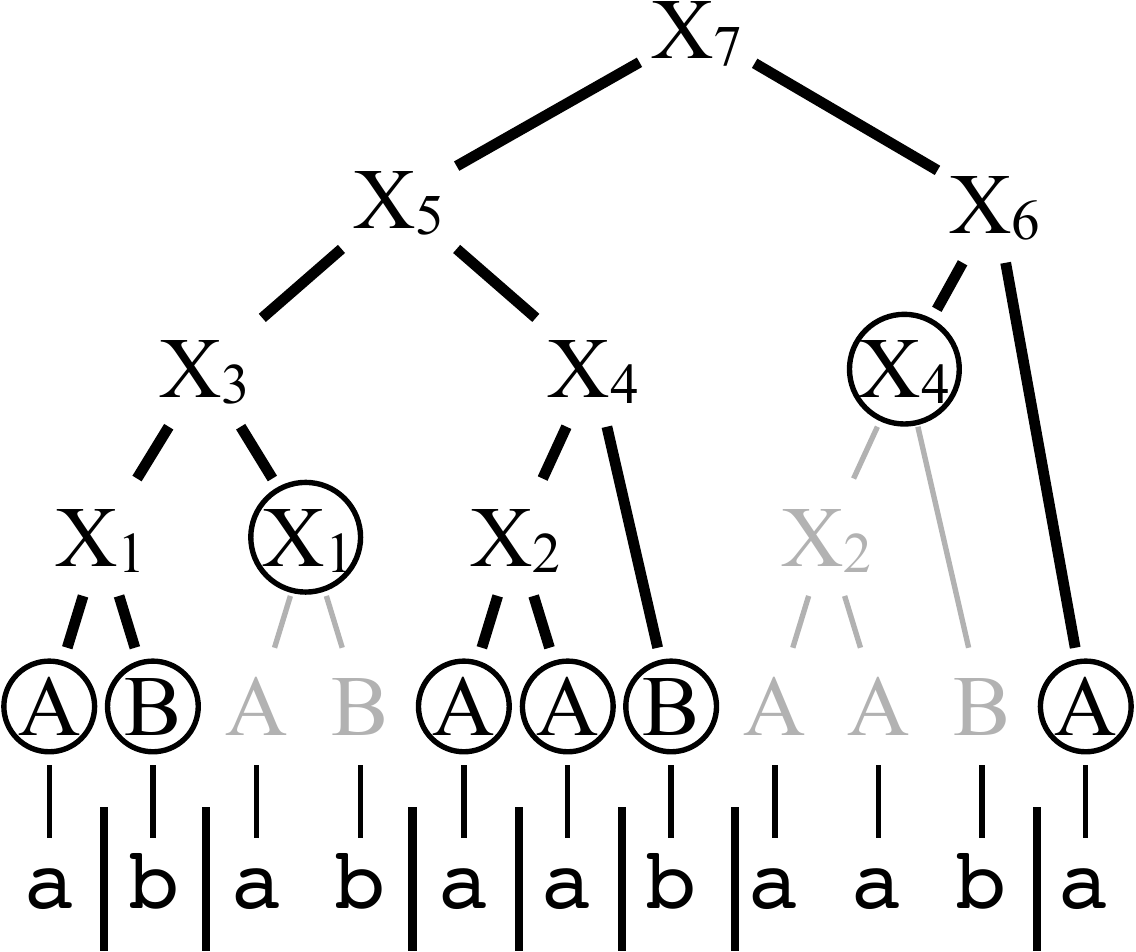}}
  \caption{
    Illustration for $\PTree(G)$ of grammar $G$ for string $w = \mathtt{ababaabaaba}$.
    The circled nodes are leaves of $\PTree(G)$.
    For this grammar $G$ of $w$, $\gfact(G) = \mathtt{a|b|ab|a|a|b|aab|a}$.
    Since $|G| = 9$, $|\gfact(G)| = 8$, and $\sigma_w = 2$,
    we can see that Lemma~\ref{lem:factsize} holds for this example.
  }
  \label{fig:gfact}
\end{figure}
It was shown in~\cite{Rytter2003LZandGrammar} that $|\gfact(G)| \le |G|$ holds for any grammar $G$.
We show a slightly tighter lower bound of $|G|$ by considering the number of distinct symbols in $w$.
\begin{lemma} \label{lem:factsize}
  For any grammar $G$ of a string $w$, $|\gfact(G)|-1+\sigma_w \le |G|$.
\end{lemma}
\begin{proof}
  A grammar $G$ in the Chomsky normal form consists of two types of productions:
  \begin{description}
    \item[Type 1] $A \rightarrow BC$ where $A, B,$ and $C$ are non-terminal symbols.
    \item[Type 2] $A \rightarrow \alpha$ where $A$ is a non-terminal symbol and $\alpha$ is a terminal symbol.
  \end{description}
  Let $g_1$ and $g_2$ be the numbers of productions of Type 1 and Type 2, respectively.
  By the definition of $\PTree(G)$, the labels of all non-leaf nodes are distinct, and they correspond to the productions of Type 1.
  Thus, the number $m$ of non-leaf nodes is at most $g_1$.
  Also, $\sigma_w \le g_2$ always holds.
  Hence, $m + \sigma_w \le g_1 + g_2 = |G|$.
  On the other hand, $m = |\gfact(G)|-1$ holds
  since $\PTree(G)$ is a full binary tree and the number of leaves of $\PTree(G)$ is $|\gfact(G)|$.
  Therefore, $|\gfact(G)|-1+\sigma_w \le |G|$ holds.
\end{proof}
We obtain the following tighter lower bound for the size of the smallest grammar(s):
\begin{theorem}\label{thm:lowerbound}
  For any string $w$, $z(w)-1+\sigma_w \le g^\ast(w)$ holds.
\end{theorem}
\begin{proof}
  It was shown in~\cite{Rytter2003LZandGrammar} that $z(w) \le |\gfact(G)|$ for any grammar $G$ of $w$.
  Thus, combining it with Lemma~\ref{lem:factsize}, we obtain the theorem.
\end{proof}
By regarding the recursive definition of $F_n$ as a grammar, we can construct a size-$n$ grammar of $F_n$.
Also, $g^\star(F_n)$ is at least $n$ by Theorem~\ref{thm:lowerbound} since $z(F_n) = n-1$ and $\sigma_{F_n} = 2$.
Thus, we obtain the following corollary:
\begin{corollary}\label{cor:smallest_size_of_F}
  The smallest grammar size of $F_n$ is $n$.
\end{corollary}
 \subsection{RePair Grammars are Smallest for Fibonacci Words}

By considering the inverse of morphism $\phi$, we have the next observation:
\begin{observation}\label{obs:fib_inverse_morphism}
  By replacing all occurrences of $\mathtt{ab}$ in $F_{i}^{(\tta, \ttb)}$ with $X$, we obtain $F_{i-1}^{(X, \tta)}$.
\end{observation}
The next lemma shows how $F$, $P$, and $Q$ can be obtained from one of the others:
\begin{lemma}\label{lem:relationFPQ}
  $\ReversePhi{\ttb}{\mathtt{ba}}(P_i) = F_{2i}$,
  $\ReversePhi{\ttb}{\mathtt{ab}}(P_i) = Q_i$, and
  $\ReversePhi{\tta}{\mathtt{ab}}(Q_i) = P_{i+1}$ hold.
\end{lemma}
\begin{proof}
  Let
  $\psi_1 = \ReversePhi{\ttb}{\mathtt{ba}}$, $\psi_2 = \ReversePhi{\ttb}{\mathtt{ab}}$, and $\psi_3 = \ReversePhi{\tta}{\mathtt{ab}}$.
  First, we consider compositions of these morphisms.
  Since
  $\phi^2(\psi_1(\tta)) = \phi^2(\tta) = \phi(\mathtt{ab}) = \mathtt{aba}$,
  $\phi^2(\psi_1(\ttb)) = \phi^2(\mathtt{ba}) = \phi(\mathtt{aab}) = \mathtt{ababa}$,
  $\psi_1(\pi(\tta)) = \psi_1(\mathtt{ab}) = \mathtt{aba}$, and
  $\psi_1(\pi(\ttb)) = \psi_1(\mathtt{abb}) = \mathtt{ababa}$,
  we have $\phi^2\circ\psi_1 = \psi_1\circ\pi$.
  Also, since
  $\psi_2(\pi(\tta)) = \psi_2(\mathtt{ab}) = \mathtt{aab}$,
  $\psi_2(\pi(\ttb)) = \psi_2(\mathtt{abb}) = \mathtt{aabab}$,
  $\theta(\psi_2(\tta)) = \theta(\tta) = \mathtt{aab}$, and
  $\theta(\psi_2(\ttb)) = \theta(\mathtt{ab}) = \mathtt{aabab}$,
  we have $\psi_2\circ\pi = \theta\circ\psi_2$.
  Also, since
  $\psi_3(\theta(\tta)) = \psi_3(\mathtt{aab}) = \mathtt{ababb}$,
  $\psi_3(\theta(\ttb)) = \psi_3(\mathtt{ab}) = \mathtt{aab}$,
  $\pi(\psi_3(\tta)) = \pi(\mathtt{ab}) = \mathtt{ababb}$, and
  $\pi(\psi_3(\ttb)) = \pi(\ttb) = \mathtt{abb}$,
  we have $\psi_3\circ\theta = \pi\circ\psi_3$.

  When $i = 1$, the lemma clearly holds.
  We assume that the lemma holds for $i-1$ with $i \ge 2$.
  Then, $\psi_1(P_i) = \psi_1(\pi(P_{i-1})) = \phi^2(\psi_1(P_{i-1})) = \phi^2(F_{2i-2}) = F_{2i}$,
  $\psi_2(P_i) = \psi_2(\pi(P_{i-1})) = \theta(\psi_2(P_{i-1})) = \theta(Q_{i-1}) = Q_i$, and
  $\psi_3(Q_i) = \psi_3(\theta(Q_{i-1})) = \pi(\psi_3(Q_{i-1})) = \pi(P_{i}) = P_{i+1}$.
\end{proof}
By considering the inverses of the three morphisms in Lemma~\ref{lem:relationFPQ},
we have the next corollary:
\begin{corollary}\label{cor:repair}
  Let $X$ denote a fresh non-terminal symbol. 
  By replacing all occurrences of $\mathtt{ba}$ in $F_{2i}^{(\tta, \ttb)}$ with $X$, we obtain $P_i^{(\tta, X)}$.
  By replacing all occurrences of $\mathtt{ab}$ in $Q_i^{(\tta, \ttb)}$ with $X$, we obtain $P_i^{(\tta, X)}$.
  By replacing all occurrences of $\mathtt{ab}$ in $P_{i+1}^{(\tta, \ttb)}$ with $X$, we obtain $Q_i^{(X, \ttb)}$.
\end{corollary}
We show examples of two RePair grammars of $F_7^{(\tta, \ttb)}$ in Figure~\ref{fig:FibRaPair}.
\begin{figure}[tb]
  \centerline{
    \includegraphics[width=0.4\linewidth]{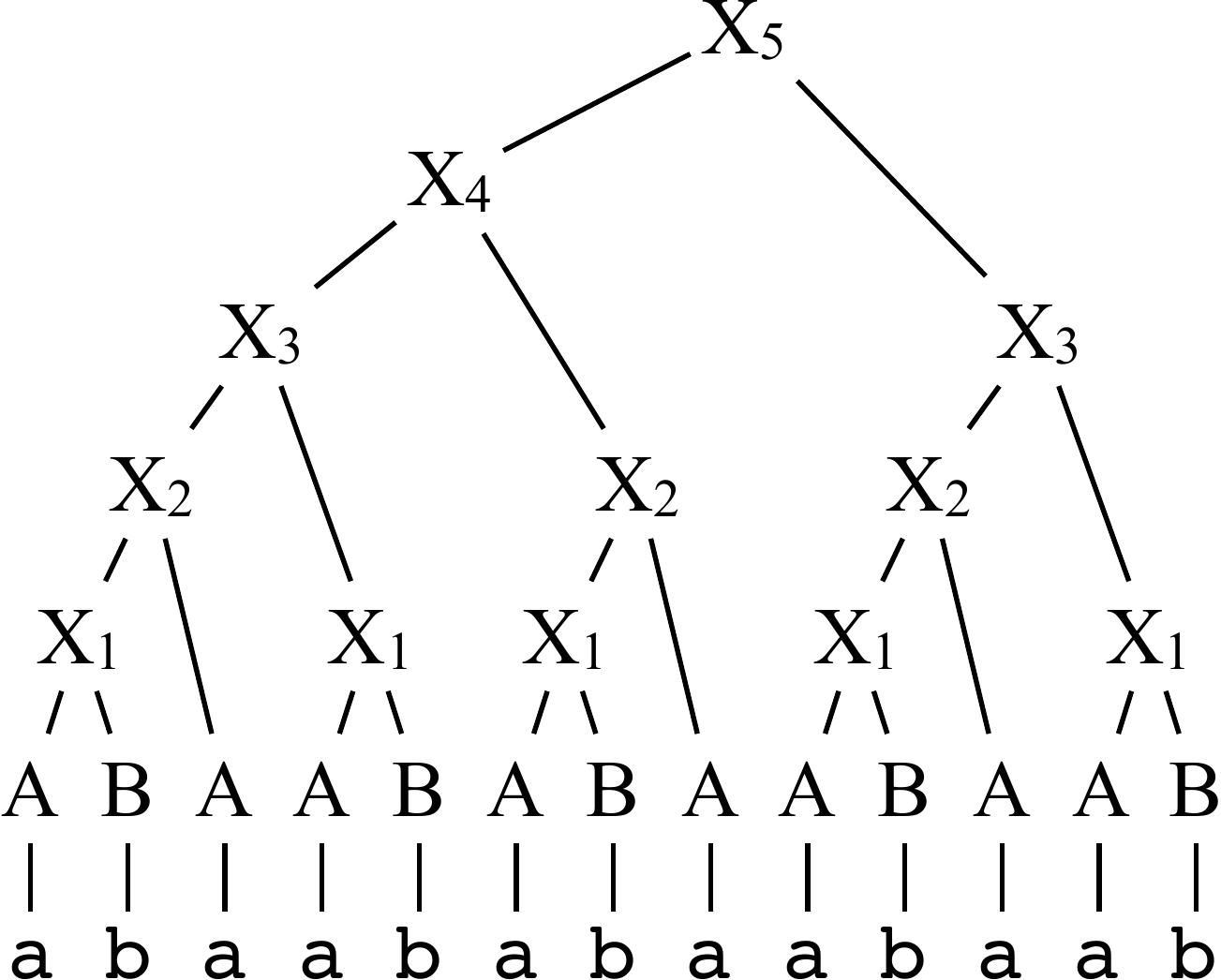}
    \hspace{0.05\linewidth}
    \includegraphics[width=0.4\linewidth]{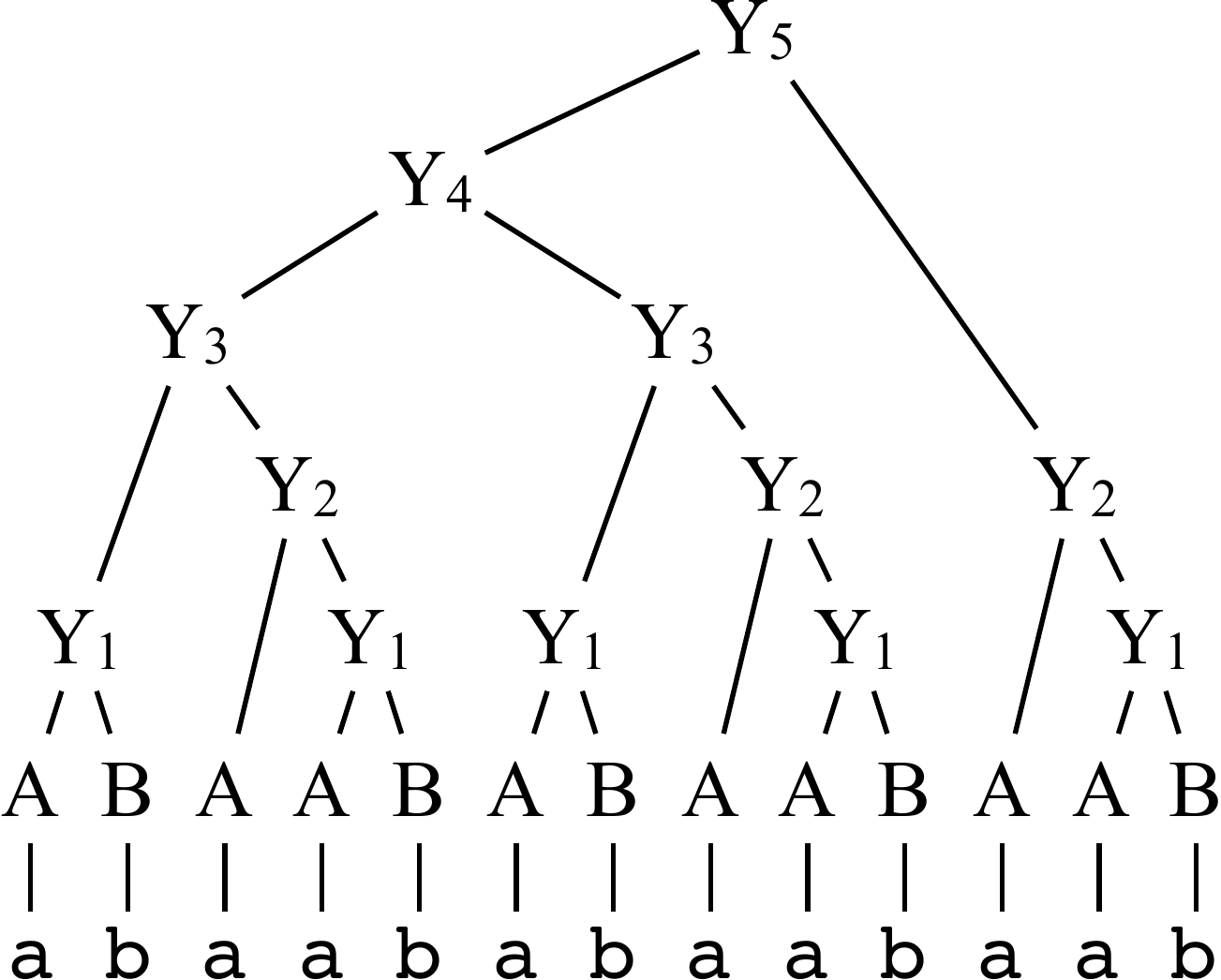}
  }
  \caption{
    Two RePair grammars of the $7$-th Fibonacci word $\mathtt{abaababaabaab}$ over $\{\tta, \ttb\}$.
  }
  \label{fig:FibRaPair}
\end{figure}

We are ready to clarify the shape of all the RePair grammars of Fibonacci words.
\begin{lemma}\label{lem:repair_is_smallest}
  The size of every RePair grammar of $F_n$ is $n$, i.e., $\RePair(F_n) \subseteq \opt(F_n)$.
  Also, $|\RePair(F_n)|=2\lfloor n/2\rfloor-2$.
\end{lemma}
\begin{proof}
  By Lemma~\ref{lem:freq}, Observation~\ref{obs:fib_inverse_morphism} and Corollary~\ref{cor:repair},
  each string, that appears while (an implementation of) the RePair algorithm is running, is one of $F$, $P$, and $Q$ over some binary alphabet.
  The change of the strings can be represented by a directed graph $(V, E)$ such that
  $V = \{F_i\mid 4 \le i \le n\}\cup\{P_i\mid 3\le i \le \lfloor n/2 \rfloor\}\cup\{Q_i\mid 2\le i \le \lfloor n/2\rfloor-1\}$
  and $E = \{(F_i, F_{i-1})\mid 5 \le i \le n\}\cup\{(F_{2k}, P_k)\mid 3\le k \le \lfloor n/2 \rfloor\}\cup\{(P_i, Q_{i-1})\mid 3 \le i \le \lfloor n/2 \rfloor\}$.
  See Figure~\ref{fig:RePairgraph} for an illustration of the graph.
  Each edge represents a replacement of all occurrences of a most frequent bigram, and thus
  each path from source ($F_n$) to sinks ($F_4$ and $Q_2$) corresponds to a RePair grammar of $F_n$.
The size of a RePair grammar is the number of edges in its corresponding path plus \emph{four},
  since the size of a minimal\footnote{This means that there are no redundant productions.} grammar of length-3-binary string, such as $F_4$ and $Q_2$, is four.
  Since the number of edges in any source-to-sinks paths is $n-4$, the size of each RePair grammar of the $n$-th Fibonacci word is $n$.
Also, the number of the RePair grammars is \emph{twice} the number of distinct source-to-sinks paths
since there are exactly two possible minimal grammars of any length-3 string.
  
  Next, let us count the number of source-to-sinks paths in the graph.
  There is only one path from $F_n$ to $F_4$, and
  there are $\lfloor n/2 \rfloor-2$ edges from $F_{2k}$ on the upper part to $P_k$ on the lower part for all $k$ with $3\le k \le \lfloor n/2\rfloor$.
  Thus, the number of distinct source-to-sinks paths is $\lfloor n/2 \rfloor-1$.
  Therefore, the number of distinct RePair grammars is $2\lfloor n/2 \rfloor-2$.
\end{proof}
\begin{figure}[tb]
  \centerline{\includegraphics[width=0.8\linewidth]{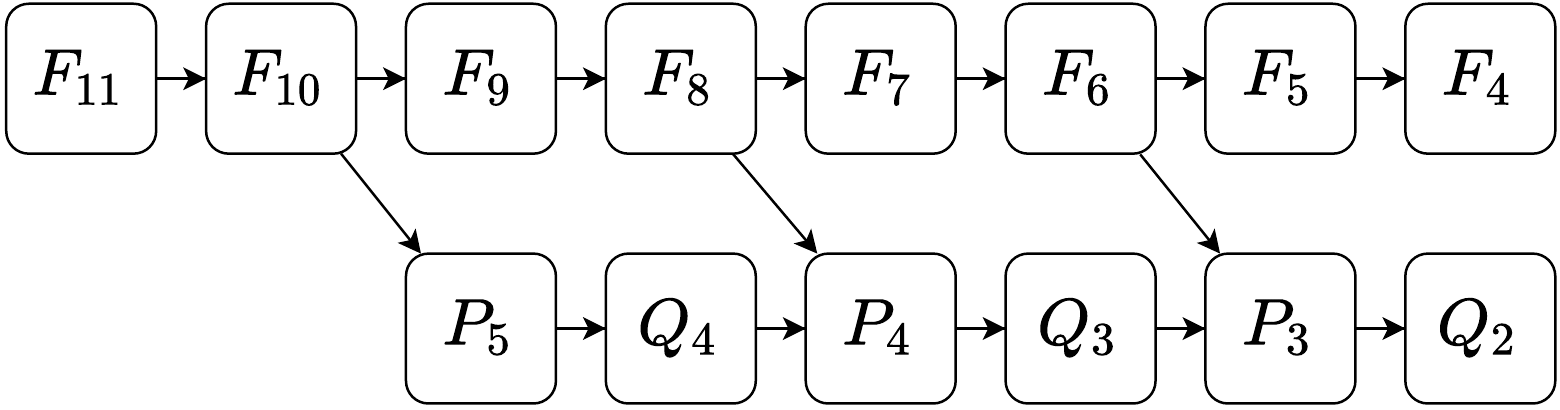}}
  \caption{
    An example of the graph for $n =11$ described in Lemma~\ref{lem:repair_is_smallest}.
  }
  \label{fig:RePairgraph}
\end{figure}
 \section{Optimality of RePair for Fibonacci Words}\label{sec:nonopt}
In this section, we prove our main theorem:
\begin{theorem}\label{thm:main}
  $\opt(F_n) = \RePair(F_n)$.
\end{theorem}

The derivation tree of any grammar (i.e., SLP) $G$ is a full binary tree.
Thus, 
there exists a \emph{bottom-up} algorithm which constructs the grammar $G$ by replacing bigrams with a non-terminal symbol one by one.
Thus, it suffices to consider all such algorithms in order to show the optimality of RePair for $F_n$.
We show that any bigram-replacement that does not satisfy the condition of RePair
always produces a larger grammar than the RePair grammars.
For $F_n$, $P_n$, and $Q_n$, there are 16 strategies that do not satisfy the condition of RePair:
\begin{table*}[ht]
\centering
\begin{tabular}{c|c|c|c|c|c|c|c|c}
  Bigram to replace & \multicolumn{2}{c|}{$\mathtt{aa}$}  & \multicolumn{2}{c|}{$\mathtt{ab}$}  & \multicolumn{2}{c|}{$\mathtt{ba}$} & \multicolumn{2}{c}{$\mathtt{bb}$}\\
  \cline{2-9}
  (all/not all of them) & all & not all & all & not all & all & not all & all & not all \\
  \hline
  $F_{2k}$ & \multirow{2}{*}{2} & \multirow{2}{*}{3} & RePair & \multirow{2}{*}{1} & RePair & 4 & - & - \\
  \cline{1-1}\cline{4-4}\cline{6-9}
  $F_{2k+1}$ & &  & RePair & & 5 & 6 & - & - \\
  \hline
  $P_n$ & - & - & RePair & 7 & 8 & 9 & 10 & 11 \\
  \hline
  $Q_n$ & 15 & 16 & RePair & 12 & 13 & 14 & - & - \\
  \hline
\end{tabular}
\end{table*}\\
The case numbers (1--16) are written inside their corresponding cells in the table.
Each hyphen shows the case where the bigram does not occur in the string,
which therefore does not need to be considered.

In order to show the non-optimality of each of the above strategies,
we utilize the sizes of LZ-factorizations which are lower bounds of the sizes of grammars.
Let $R$ be the string obtained by replacing occurrences of a bigram in $F_n$ with a non-terminal symbol $X$ by one of the above 16 strategies.
We will show that $z(R) \ge n-1$ holds for each case.
Then, by Theorem~\ref{thm:lowerbound}, the size of the corresponding grammar of $F_n$ becomes at least
$(z(R) + |\{X\}| - 1) + \sigma_{F_n} \ge (n-1)+2 = n+1$, i.e., that is not the smallest by Corollary~\ref{cor:smallest_size_of_F}.

To compare the LZ-factorizations between two strings transformed from the same string $F_n$,
we treat the boundaries as if they are on $F_n$.

\subsection{Non-optimality of Strategies for $F_n$}

We first define a \emph{semi-greedy} factorization $\SG(w)$ of string $w$ which will be used in the proof for the first three cases.
Let $\SG(w)$ be the factorization of $w$ obtained
by shifting each boundary of $LZ(w)$ except the ones whose left phase is of length $1$
to the left by one.
For example,
$\SG(F_7) = \mathtt{a|b|a|ab|abaab|aab}$
since
$\LZ(F_7) = \mathtt{a|b|a|aba|baaba|ab}$.
Clearly, $|\SG(F_n)| = |\LZ(F_n)| = n-1$.
By the definition of $\SG(F_n)$ and properties of $\LZ(F_n)$~(cf.~\cite{Berstel_2006,Fici2015}), the following claim holds:
\begin{claim}\label{claim:SG}
  Let $\SG(F_n) = (p_1, \ldots, p_{n-1})$ for $n \ge 5$.
  The following statements hold:
  \begin{itemize}
    \item The first four phrases are $(p_1, p_2, p_3, p_4) = (\tta, \ttb, \tta, \mathtt{ab})$.
    \item For each $i$ with $5 \le i \le n-2$, $p_i$ is the right-rotation of $F_i^R$ and it is a greedy phrase.
    \item The last phase is $p_{n-1} = \mathtt{aba}$ if $n$ is even, and $p_{n-1} = \mathtt{aab}$ otherwise.
    \item Each boundary of $\SG(F_n)$, except the first and third ones, divides an occurrence of $\mathtt{ba}$.
  \end{itemize}
\end{claim}

\paragraph*{\bf Case (1): Replacing some but not all the occurrences of $\mathtt{ab}$ in $F_n$.}
Recall that $F_{n-1}^{(X, \tta)}$ is obtained by replacing all the occurrences of $\mathtt{ab}$ in $F_n^{(\tta, \ttb)}$ with $X$.
Let $R_1^{(X, \tta, \ttb)}$ be any string obtained by replacing some but \emph{not all} the occurrences of $\mathtt{ab}$ in $F_n^{(\tta, \ttb)}$ with $X$.
Let $\LZ(F_{n-1}^{(X, \tta)}) = (p_1, \ldots, p_{n-2})$ and $\LZ(R_1^{(X, \tta, \ttb)}) = (q_1, \ldots, q_{z'})$ where $z' = |\LZ(R_1^{(X, \tta, \ttb)})|$.
See Figure~\ref{fig:LZK1} for illustration.
\begin{figure}[tb]
  \centerline{\includegraphics[width=\linewidth]{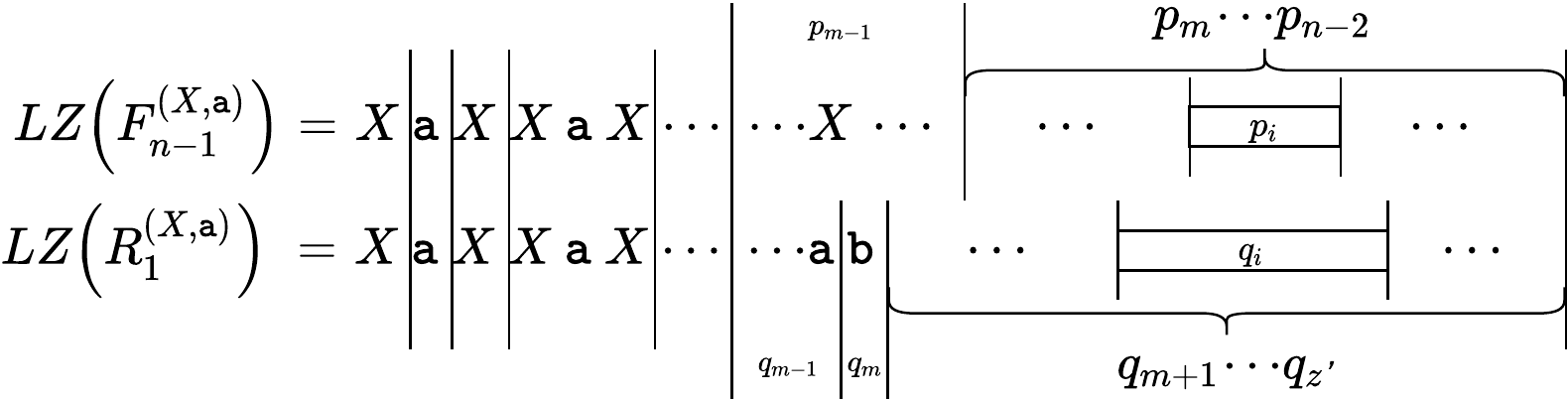}}
  \caption{
    Illustration for contradiction of $\LZ(F_{n-1}^{(X, \tta)}) = (p_1, \ldots, p_{n-2})$ and $\LZ(R_1^{(X, \tta, \ttb)}) = (q_1, \ldots, q_{z'})$ for Case (1).
    Note that the scale of this figure is based on the length of $F_n \in \{\tta, \ttb\}^\star$, not the lengths of phrases.
  }
  \label{fig:LZK1}
\end{figure}
The first mismatch of boundaries between two factorizations is the position of the first occurrence of $\ttb$ in $R_1^{(X, \tta, \ttb)}$.
Since the $\ttb$ is a fresh symbol, it is a length-$1$ phrase.
Suppose that this length-$1$ phrase is the $m$-th phrase ($m\ge 2$) in $\LZ(R_1^{(X, \tta, \ttb)})$.
Then, $\ReversePhi{X}{\mathtt{ab}}(p_{m-1}\cdots p_{n-2}) = \ReversePhi{X}{\mathtt{ab}}(q_{m-1}\cdots q_{z'})$ holds.
The next corollary follows from Claim~\ref{claim:SG}:
\begin{corollary} \label{cor:reverse_phi_eq_SG}
  The factorization $\ReversePhi{X}{\mathtt{ab}}(\LZ(F_{n-1}^{(X, \tta)}))$ of $F_n^{(\tta, \ttb)}$ is the same as $\SG(F_n^{(\tta, \ttb)})$ except the first phrase.
  In other words, for each $i$ with $2\le i \le n-2$, $\ReversePhi{X}{\mathtt{ab}}(p_i)$ is the $(i+1)$-th phrase of $\SG(F_n^{(\tta, \ttb)})$.
\end{corollary}
From the greediness of $\ReversePhi{X}{\mathtt{ab}}(p_{m-1})$ in $\SG(F_n^{(\tta, \ttb)})$, $p_{m-1}$ is not shorter than $q_{m-1}$.
Thus, $\ReversePhi{X}{\mathtt{ab}}(p_m\cdots p_{n-2})$ is not longer than $\ReversePhi{X}{\mathtt{ab}}(q_{m+1}\cdots q_{z'})$.
For the sake of contradiction, we assume that $z' < n-1$.
Then, $z'-(m+1)+1< (n-2)-m+1$ holds, and hence, there must exist a phrase $q_i$ of $\LZ(R_1^{(X, \tta, \ttb)})$ and a phrase $p_j$ of $\LZ(F_{n-1}^{(\tta, \ttb)})$ such that
$\ReversePhi{X}{\mathtt{ab}}(q_i)$ contains $\ReversePhi{X}{\mathtt{ab}}(p_j)$ and their ending positions in $F_n^{(\tta, \ttb)}$ are different.
This contradicts the greediness of the phrase $\ReversePhi{X}{\mathtt{ab}}(p_j)$ of $\SG(F_n^{(\tta, \ttb)})$ on $F_n^{(\tta, \ttb)}$.
Therefore, $|\LZ(R_1^{(X, \tta, \ttb)})| = z' \ge n-1$.

Basically, most of the remaining cases can be proven by similar argumentations,
however, we will write down the details because there are a few differences.

\paragraph*{\bf Case (2): Replacing all the occurrences of $\mathtt{aa}$ in $F_n$.}
Let $R_2^{(X, \tta, \ttb)}$ be the string obtained by replacing all the occurrences of $\mathtt{aa}$ in $F_n^{(\tta, \ttb)}$ with $X$.
The next corollary holds from Claim~\ref{claim:SG}~(see also Figure~\ref{fig:SGandLZ} for a concrete example):
\begin{corollary}\label{cor:K2_F}
  The factorization $\ReversePhi{X}{\mathtt{aa}}(\LZ(R_2^{(X, \tta, \ttb)}))$ of $F_n^{(\tta, \ttb)}$ is the same as $\SG(F_n^{(\tta, \ttb)})$ except the first four phrases.
  In other words, for each $i$ with $5\le i \le n-1$, $\ReversePhi{X}{\mathtt{aa}}(p_i)$ is the $i$-th phrase of $\SG(F_n^{(\tta, \ttb)})$,
  where $p_i$ is the $i$-th phrase of $\LZ(R_2^{(X, \tta, \ttb)})$.
\end{corollary}
\begin{figure}[tb]
  \centerline{\includegraphics[width=\linewidth]{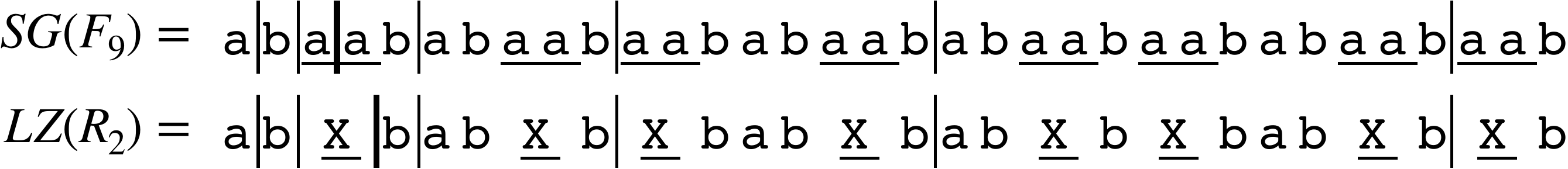}}
  \caption{
    Two factorizations $\SG(F_9)$ and $\LZ(R_2)$. 
  }
  \label{fig:SGandLZ}
\end{figure}
Thus, $|\LZ(R_2^{(X, \tta, \ttb)})| = |\SG(F_n^{(\tta, \ttb)})| = n-1$.

\paragraph*{\bf Case (3): Replacing some but not all the occurrences of $\mathtt{aa}$ in $F_n$.}
Let $R_3^{(X, \tta, \ttb)}$ be any string obtained by replacing some but {not} all the occurrences of $\mathtt{aa}$ in $F_n^{(\tta, \ttb)}$ with $X$.
Let $\LZ(R_2^{(X, \tta, \ttb)}) = (p_1, \ldots, p_{n-1})$ and $\LZ(R_3^{(X, \tta, \ttb)}) = (q_1, \ldots, q_{z'})$ where $z' = |\LZ(R_3^{(X, \tta, \ttb)})|$.
We omit the superscripts in the following.
The first mismatch of boundaries between $\LZ(R_2)$ and $\LZ(R_3)$
is the position of the first occurrence of $\mathtt{aa}$ in $R_3$.
Since this is the first occurrence of $\mathtt{aa}$, there has to be a boundary between the two $\mathtt{a}$'s.
Suppose that the phrase that starts with the second $\mathtt{a}$ is the $m$-th phrase ($m\ge 4$) in $\LZ(R_3)$.
By Corollary~\ref{cor:K2_F}, $p_{m-1}$ is not shorter than $q_{m-1}$.
Thus, $\ReversePhi{X}{\mathtt{aa}}(q_m\cdots q_{z'})$ is longer than $\ReversePhi{X}{\mathtt{aa}}(p_m\cdots p_{n-1})$.
For the sake of contradiction, we assume that $z' < n-1$.
Then, $z'-m+1 < (n-1)-m+1$ holds, and hence, there exist phrases $q_i$ of $\LZ(R_3)$ and $p_j$ of $\LZ(R_2)$ such that
$\ReversePhi{X}{\mathtt{aa}}(q_i)$ contains $\ReversePhi{X}{\mathtt{aa}}(p_j)$ and their ending positions in $F_n$ are different.
This contradicts the greediness of the phrase $\ReversePhi{X}{\mathtt{aa}}(p_j)$ of $\SG(F_n)$.
Therefore, $|\LZ(R_3)| = z' \ge n-1$.

\paragraph*{\bf Case (4): Replacing some but not all the occurrences of $\mathtt{ba}$ in $F_{2k}$.}
Recall that $P_k^{(\tta, X)}$ is obtained by replacing all the occurrences of $\mathtt{ba}$ in $F_{2k}^{(\tta, \ttb)}$ with $X$.
Let $R_4^{(X, \tta, \ttb)}$ be any string obtained by replacing some but not all the occurrences of $\mathtt{ba}$ in $F_{2k}^{(\tta, \ttb)}$ with $X$.
Let $\LZ(P_k^{(\tta, X)}) = (p_1, \ldots, p_{2k-2})$ and $\LZ(R_4^{(X, \tta, \ttb)}) = (q_1, \ldots, q_{z'})$ where $z' = |\LZ(R_4^{(X, \tta, \ttb)})|$.
Since the only boundary in $\LZ(F_{2k}^{(\tta, \ttb)})$ that divides an occurrence of $\mathtt{ba}$ is the second one,
the next holds for the LZ-factorization of $P_k^{(\tta, X)}$~(see also Figure~\ref{fig:LZP} for a concrete example):
\begin{corollary}\label{cor:P_F}
  The factorization $\ReversePhi{X}{\mathtt{ba}}(\LZ(P_k^{(\tta, X)}))$ of $F_n^{(\tta, \ttb)}$ is the same as $\LZ(F_{2k}^{(\tta, \ttb)})$ except the first three phrases.
  In other words, for each $i$ with $4\le i \le 2k-2$, $\ReversePhi{X}{\mathtt{ba}}(p_i)$ is the $(i+1)$-th phrase of $\LZ(F_{2k}^{(\tta, \ttb)})$.
\end{corollary}
\begin{figure}[tb]
  \centerline{\includegraphics[width=0.6\linewidth]{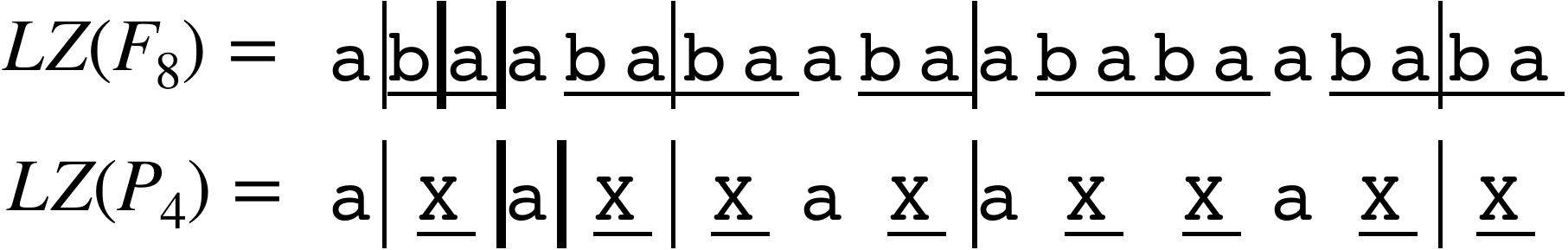}}
  \caption{
    Two factorizations $\LZ(F_8)$ and $\LZ(P_4)$. 
  }
  \label{fig:LZP}
\end{figure}
We omit the superscripts in the following.
The first mismatch of boundaries between two factorizations is the position of the first occurrence of $\ttb$ in $R_4$.
Since the $\ttb$ is a fresh symbol, it is a length-$1$ phrase.
Let the length-$1$ phrase be the $m$-th phrase in $\LZ(R_4)$.
Then, $\ReversePhi{X}{\mathtt{ba}}(q_m\cdots q_{z'})$ is longer than $\ReversePhi{X}{\mathtt{ba}}(p_m\cdots p_{2k-2})$ by Corollary~\ref{cor:P_F}.
For the sake of contradiction, we assume that $z' \le 2k-2$.
Then $z'-m+1 \le (2k-2)-m+1$, and hence, there exist phrases $q_i$ of $\LZ(R_4)$ and $p_j$ of $\LZ(P_k)$ such that
$\ReversePhi{X}{\mathtt{ba}}(q_i)$ contains $\ReversePhi{X}{\mathtt{ba}}(p_j)$ and their ending positions in $F_{2k}$ are different.
This contradicts that the greediness of phrase $\ReversePhi{X}{\mathtt{ba}}(p_j)$ of $\LZ(F_{2k})$,
Therefore, $|\LZ(R_4)| = z' > 2k-2$.

\paragraph*{\bf Case (5): Replacing all the occurrences of $\mathtt{ba}$ in $F_{2k+1}$.}
Let $R_5^{(X, \tta, \ttb)}$ be the string obtained by replacing all the occurrences of $\mathtt{ba}$ in $F_{2k+1}^{(\tta, \ttb)}$ with $X$.
Since $F_{2k+1}^{(\tta, \ttb)}$ ends with $\ttb$, the last symbol of $R_5^{(X, \tta, \ttb)}$ is $\ttb$ and it is unique in $R_5^{(X, \tta, \ttb)}$.
We omit the superscripts in the following.
Since $F_{2k+1} = F_{2k}F_{2k-2}F_{2k-3}$, $P_kP_{k-1}$ is a prefix of $R_5$.
By Lemma~\ref{lem:lz_of_P}, the first $2k-2$ phrases of $\LZ(R_5)$ is the same as that of $\LZ(P_{k+1})$.
Also, the $(2k-1)$-th phrase ends at before $\ttb$ and the $2k$-th phrase is $\ttb$.
Thus, $|\LZ(R_5)| = 2k$.

\paragraph*{\bf Case (6): Replacing some but not all the occurrences of $\mathtt{ba}$ in $F_{2k+1}$.}
Let $R_6^{(X, \tta, \ttb)}$ be any string obtained by replacing some but not all the occurrences of $\mathtt{ba}$ in $F_{2k+1}^{(\tta, \ttb)}$ with $X$.
Let $\LZ(R_5^{(X, \tta, \ttb)}) = (p_1, \ldots, p_{2k})$ and $\LZ(R_6^{(X, \tta, \ttb)}) = (q_1, \ldots, q_{z'})$ where $z' = |\LZ(R_6^{(X, \tta, \ttb)})|$.
We omit the superscripts in the following.
The first mismatch of boundaries between two factorizations is the position of the first occurrence of $\ttb$ in $R_6$.
Since the $\ttb$ is a fresh symbol, it is a length-$1$ phrase.
Let the length-$1$ phrase be the $m$-th phrase in $\LZ(R_6)$.
Then, $\ReversePhi{X}{\mathtt{ba}}(p_m\cdots p_{2k})$ is not longer than $\ReversePhi{X}{\mathtt{ba}}(q_m\cdots q_{z'})$
by the greediness of $\ReversePhi{X}{\mathtt{ba}}(p_{m-1})$.
For the sake of contradiction, we assume that $z' < 2k$.
Then $z'-m+1 < 2k-m+1$, and hence, there exist phrases $q_i$ of $\LZ(R_6)$ and $p_j$ of $\LZ(R_5)$ such that
$\ReversePhi{X}{\mathtt{ba}}(q_i)$ contains $\ReversePhi{X}{\mathtt{ba}}(p_j)$ and their ending positions in $F_{2k+1}$ are different.
This contradicts the greediness of phrase $\ReversePhi{X}{\mathtt{ba}}(p_j)$ of $\LZ(F_{2k+1})$.
Therefore, $|\LZ(R_6)| = z' \ge 2k$.

The proofs for the remaining ten cases can be found in a full version of this paper~\cite{MienoIH2022_smallest_grammar}.
We remark that the remaining ten cases can also be proven by similar argumentations.
 \section{Conclusions}
In this paper, we analyzed the smallest grammars of Fibonacci words
and completely characterized them by the RePair grammar-compressor.
Namely, the set of all smallest grammars that produce only the $n$-th Fibonacci word $F_n$ equals the set of all grammars obtained by applying (different implementations of) the RePair algorithm to $F_n$.
Further, we showed that the size of the smallest grammars of $F_n$ is $n$ and that
the number of such grammars is $2\lfloor n/2 \rfloor -2$.

To show the smallest grammar size of $F_n$,
we revisited the result on the lower bound of the sizes of grammars shown by Rytter~\cite{Rytter2003LZandGrammar}.
Here, we gave a slightly tighter lower bound of the grammar size $z(w) - 1 + \sigma_w$ for \emph{any} string $w$.
Independent of the above results on Fibonacci words,
this result on a lower bound is interesting since the result will help
show the \emph{exact} values of the smallest grammar size of other strings.

It is left as our future work to investigate whether it is possible to characterize the smallest grammars of other binary words, such as Thue-Morse words and Period-doubling words, by similar methods to Fibonacci words.
 \section*{Acknowledgements}
This work was supported by JSPS KAKENHI Grant Numbers JP20J11983 (TM), 20H05964 (TM, TH) and JST PRESTO Grant Number JPMJPR1922 (SI).

\end{document}